\documentclass[11pt,a4paper]{article}

\usepackage{textcomp}
\usepackage{graphicx}
\usepackage{a4wide}
\usepackage{graphicx}
\usepackage{authblk}
\usepackage{amssymb}
\usepackage{epstopdf}
\usepackage[sans]{dsfont}
\usepackage[applemac]{inputenc}
\usepackage[english]{babel}
\usepackage{latexsym}
\usepackage{subfigure} 
\usepackage{multirow}
\usepackage{appendix}
\usepackage{color}
\usepackage{float}
\usepackage{amsmath}
\usepackage{amsfonts}
\numberwithin{equation}{section}
\usepackage{amsthm}
\usepackage[numbers]{natbib}
\usepackage[bookmarks=true,colorlinks=true,linkcolor={blue},urlcolor={blue}, citecolor={blue},pdfstartview={XYZ null null 1.22}]{hyperref}%

\makeindex             

\newtheorem{remark}{Remark}
\newtheorem{proposition}{Proposition}
\newtheorem{theorem}{Theorem}

\newtheorem{lemma}{Lemma}

\def\RR{\mathbb R}
\def\be{\begin{equation}}
\def\ee{\end{equation}}
\def\bea{\begin{eqnarray}}
\def\eea{\end{eqnarray}}

\begin{document}

\title{Structure preserving schemes for mean-field equations of collective behavior}
\author{Lorenzo Pareschi\thanks{Department of Mathematics and Computer Science, University of Ferrara, Via Machiavelli 35, 44121 Ferrara, Italy ({\tt lorenzo.pareschi@unife.it}).} \and
Mattia Zanella\thanks{Department of Mathematical Sciences, Politecnico di Torino, Corso degli Abruzzi 24, 10129, Torino, Italy ({\tt mattia.zanella@polito.it}).}
}

\maketitle

\abstract{In this paper we consider the development of numerical schemes for mean-field equations describing the collective behavior of a large group of interacting agents. The schemes are based on a generalization of the classical Chang-Cooper approach and are capable to preserve the main structural properties of the systems, namely nonnegativity of the solution, physical conservation laws, entropy dissipation and stationary solutions. In particular, the methods here derived are second order accurate in transient regimes whereas they can reach arbitrary accuracy asymptotically for large times. Several examples are reported to show the generality of the approach.}


\section{Introduction}
The description of social dynamics characterized by emerging collective behaviors has gained increasing popularity in the recent years \cite{APTZ, BAMT, CFTV, CS, DOCBC, PT0}. Typical examples are groups of animals/humans with a tendency to flock or herd but also interacting {agents in a financial market}, potential {voters during political elections} and connected {members of a social network}.  

In the mathematical description classical particles are replaced by more complex structures ({agents}, {active particles},...) which take into account additional aspects related to the various specific fields of application, like behavioral characteristics, visual perception, experience/knowledge and so on. Various {microscopic models} have been introduced in different communities with the aim to reproduce qualitatively the dynamics and to capture some essential {stylized facts} (clusters, power laws, consensus, flocking, ...). 

In spite of many differences between classical particle dynamics and systems of interacting agents (equation are not a consequence of fundamental physical laws derived from first principles) one can apply {similar methodological approaches}. In particular, to analyze the formation of stylized facts and reduce the computational complexity of the agents' dynamics, it is of utmost importance to derive the corresponding {mesoscopic/kinetic} description \cite{APTZ,APZ3,BD,CFRT,CFTV,FPTT,PT0,PT1}.

These kinetic equations are derived in the limit of a large number of interacting agents and describe the evolution of a non negative distribution function $f(x,w,t)$, $t\ge 0$, $x\in  \RR^{d_x}$, $w\in  \RR^{d_w}$, $d_x,d_w\ge 1$, which satisfies a mean-field type equation of the general form
\begin{equation}
\label{eq:NAD}
\partial_t f + {\mathcal L}[f]= \nabla_w \cdot \Big[ \mathcal B[f]f + \nabla_w (Df) \Big],
\end{equation}
where ${\mathcal L}[\cdot](x,w,t)$ is an operator describing the agents' dynamics with respect to the $x$-variable, $\mathcal B[\cdot](x,w,t)$ is an alignment operator in the $w$-variable and $D=D(x,w)\ge0$ is a diffusion function. 

The most celebrated example is given by the mean-field \emph{Cucker-Smale} model \cite{CFRT, CFTV, CS, PT0} which, in absence of diffusion, corresponds to the choices 
\be
{\mathcal L}[f] = w\cdot \nabla_x f,\qquad {\mathcal B}[f]=\int_{\RR^{d_v}\times \RR^{d_x}} H(x,y)(w-v)f(y,v,t)\,dy\,dv,
\label{eq:CS}
\ee
where 
\be
H(x,y)=\frac1{(1+(x-y)^2)^\gamma},\qquad \gamma \geq 0.
\label{eq:csH}
\ee 
The model describes the alignment process in a multidimensional group of agents (birds, insects, \ldots), when all agents are aligned with equal speed a flocking state is reached. For the above choice of $H$ it has been proved that if $\gamma \leq 1/2$, independently on their initial state, all agents tend to move exponentially fast with the same velocity, while their
relative distances tend to remain constant. The addition of a diffusion term weighted by $D\in\mathbb\RR^+$ has been studied in \cite{BaDe,BoCar} among others. 

Another example is the non homogeneous mean-field \emph{Cordier-Pareschi-Toscani} model \cite{CPT, PT1} which describes the evolution of the distribution $f(x,w,t)$ of wealth $w\in \RR^+$ in a set of agents with a given propensity to invest $x\in [0,1]$. In our notations it corresponds to
\be
{\mathcal L}[f] = \phi(x,w) \partial_x f,\qquad {\mathcal B}[f]=\int_{\RR^{+}}(w-v)f(y,v,t)\,dv,\qquad D=\frac{\sigma^2}{2} w^2.
\label{eq:CPT}
\ee
The equilibrium solutions in the homogeneous case, $f=f(w,t)$ independent of $x$, present the formation of {power-laws} and read
\be
f_{\infty}(w)=\frac{(\mu-1)^{\mu}}{\Gamma(\mu)w^{1+\mu}}\exp\left(-\frac{\mu-1}{w}\right),
\ee
with $\mu=1+2/\sigma^2 > 1$ the \emph{Pareto exponent} and $\int_{\RR^+} f_{\infty}(w) w\, dw =1$.

Finally, a third example is represented by the mean-field \emph{Albi-Pareschi-Zanella} model \cite{APTZ, APZ3} describing the opinion dynamics of a group of interacting agents over a social network. The evolution of the distribution $f(x,w,t)$ of agents with a given opinion $w \in [-1,1]$ and a certain amount of discrete connections $x \in \left\{0,1,\ldots,c_{max}\right\}$, is characterized by
\be
\begin{split}
{\mathcal L}[f] =& -\dfrac{2 V_r(f;w)}{\gamma+\beta}\left[(x+1+\beta)f(x+1,w,t)-(x+\beta)f(x,w,t)\right]\\
&-\dfrac{2 V_a (f;w)}{\gamma+\alpha}\left[(x-1+\alpha)f(x-1,w,t)-(x+\alpha)f(x,w,t)\right],\\
{\mathcal B}[f]=&\sum_{y=0}^{c_{max}}\int_{[-1,1]} P(w,v;x,y)(w-v)f(v,y,t)\,dv,
\end{split}
\label{eq:APZ}
\end{equation} 
where $P(\cdot,\cdot;\cdot,\cdot)\in [0,1]$ is a compromise function, 
$\gamma=\gamma(t)$ is the mean density of connectivity 
$\gamma(t) = \sum_{x=0}^{c_{max}}x\int_{[-1,1]}f(x,w,t)\,dw$,
$\alpha, \beta>0$ are attraction coefficients, and $V_r(f;w)\geq 0$, $V_a(f;w)\geq 0$ are characteristic rates of the connections removal and adding processes, respectively.

Different equilibrium solutions in the case $f=f(w,t)$ independent of $x$, are possible depending on the choices of $P$ and $D$. For example, if $P\equiv 1$ and $D=\sigma^2 (1-w^2)^2 / 2$ the steady state reads
\be
f_{\infty}(w)= C_0(1+w)^{-2+\bar m /\sigma^2}(1-w)^{-2-\bar m/\sigma^2}\exp\Big\{-\dfrac{(1-\bar m w)}{\sigma^2(1-w^2)}\Big\},
\label{eq:ops}
\ee
where $\bar m = \int_{[-1,1]} w f_{\infty}(w)\,dw$ and $C_0$ is such that $\int_{[-1,1]} f_{\infty}(w)\,dw=1$.

The development of numerical methods for the above class of equations is challenging due to the intrinsic structural properties of the solution \cite{BCS,BD,CCH,CC,gosse,LLPS,PZ}. Non negativity of the distribution function, conservation of invariant quantities (like moments in $w$ of the distribution function), entropy dissipation and homogeneous steady states are essential in order to compute qualitatively correct solutions of the mean-field equation. 

In this paper we focus on the construction of numerical methods which preserves such structural properties and in particular, which are able to capture the correct steady state of the mean-field problem with arbitrary order of accuracy. The schemes are based on a suitable generalization of the Chang-Cooper approach to nonlinear problems of Fokker-Planck type and are derived in the next Section. Their properties are then discussed in Section 3. Finally numerical results are presented in Section 4.

\section{Derivation of the schemes}\label{sec:3}
Since most of the structural properties are related to the right hand side in (\ref{eq:NAD}) in the following we will focus on the homogeneous case $f=f(w,t)$. Connections with the full problem are then recovered using splitting methods or other partitioned time discretization schemes, like additive Runge-Kutta methods \cite{HNW}.


Under this assumption, we can rewrite the mean-field equation \eqref{eq:NAD} as 
\be\label{eq:NAD2}
\partial_t f(w,t) = \nabla_w \cdot [(\mathcal B[f](w,t)+\nabla_w D(w))f(w,t)+D(w)\nabla_w f(w,t)].
\ee
We define the $d-$dimensional flux function
\be\label{eq:flux_general}
\mathcal F[f](w,t) = (\mathcal B[f](w,t)+\nabla_w D(w))f(w,t)+D(w)\nabla_w f(w,t),
\ee
so that the equation may be written in conservative form as
\be\label{eq:NAD_dim}
\partial_t f(w,t) = \nabla_w \cdot \mathcal F(w,t). 
\ee
\subsection{One-dimensional case}
Let us consider for notation simplicity the one-dimensional case
\be\label{eq:FP_flux}
\partial_t f(w,t) = \partial_w \mathcal F[f](w,t),
\ee
where
\be
\mathcal F[f](w,t) = ( \mathcal B[f](w,t)+D'(w)) f(w,t)+D(w)\partial_w f(w,t)
\ee
and we used the notation $D'(w)=\partial_w D(w)$ and assume $D(w)$ strictly positive in the internal points of the computational domain. We introduce an uniform spatial grid $w_i$, $i=0,\dots,N$ such that $w_{i+1}-w_i=\Delta w$. We denote as usual $w_{i\pm 1/2}=w_i\pm \Delta/2$ and consider the conservative discretization of equation \eqref{eq:FP_flux}
\be\label{eq:flux}
\dfrac{d}{dt} f_i(t) = \dfrac{\mathcal F_{i+1/2}[f](t)-\mathcal F_{i-1/2}[f](t)}{\Delta w},
\ee
where for each $t\ge 0$, $f_i(t)$ is an approximation of $f(w_i,t)$ and  $\mathcal F_{i\pm 1/2}[f](t)$ is the flux function characterizing the discretization. 

Let us set  $\mathcal C[f](w,t)=\mathcal B[f](w,t)+D'(w)$ and adopt the notations $\mathcal B_{i+1/2}=\mathcal B[f](w_{i+1/2},t)$, $D_{i+1/2}=D(w_{i+1/2})$, $D'_{i+1/2}=D'(w_{i+1/2})$. We will consider a general flux function which is combination of the grid points $i+1$ and $i$ as in \cite{CC, PZ}
\be\begin{split}\label{eq:CC_flux}
\mathcal F_{i+ 1/2}[f] = \tilde{\mathcal C}_{i+1/2}\tilde f_{i+1/2}+D_{i+1/2}\dfrac{f_{i+1}-f_i}{\Delta w},
\end{split}\ee
where 
\be\label{eq:f_CC}
\tilde f_{i+1/2}=(1-\delta_{i+1/2})f_{i+1}+\delta_{i+1/2}f_i.
\ee
Here, we aim at deriving suitable expressions for $\delta_{i+1/2}$ and $\tilde{\mathcal C}_{i+1/2}$ in such a way that the method yields nonnegative solutions, without restrictions on $\Delta w$, and preserves the steady state of the system with arbitrary accuracy. 

 For example, the standard approach based on central difference is obtained taking $\delta_{i+1/2}=1/2$ and  $\tilde{\mathcal{C}}_{i+1/2}={\mathcal{B}}_{i+1/2}$, $\forall\,i$. It is well-known, however, that such a discretization method is subject to restrictive conditions over the mesh size $\Delta w$ in order to keep non negativity of the solution.  

Here, we aim at deriving suitable expressions for $\delta_{i+1/2}$ and $\tilde{\mathcal{C}}_{i+1/2}$ in such a way that the method yields nonnegative solutions without restriction on $\Delta w$ and preserves the steady state of the system with arbitrary order of accuracy.

First, observe that at the steady state the numerical flux equal should vanish. From \eqref{eq:CC_flux} we get 
\be
\dfrac{f_{i+1}}{f_i} = \dfrac{-\delta_{i+1/2}\tilde{\mathcal C}_{i+1/2}+\dfrac{D_{i+1/2}}{\Delta w}}{(1-\delta_{i+1/2})\tilde{\mathcal C}_{i+1/2}+\dfrac{D_{i+1/2}}{\Delta w}}.
\ee
Similarly, if we consider the analytical flux at the steady state, we have
\be\label{eq:FP_steady}
D(w)\partial_w f(w,t) = -(\mathcal B[f]+D'(w))f(w,t),
\ee
which is in general not solvable, except in some special cases due to the nonlinearity on the right hand side. We may overcome this difficulty in the quasi steady-state approximation integrating equation \eqref{eq:FP_steady} on the cell $[w_i,w_{i+1}]$
\be
\int_{w_i}^{w_{i+1}}\dfrac{1}{f(w,t)}\partial_w f(w,t)dw = -\int_{w_i}^{w_{i+1}}\dfrac{1}{D(w)}(\mathcal B[f](w,t)+D'(w))dw,
\ee
which gives
\be\label{eq:quasi_SS}
\dfrac{f_{i+1}}{f_i} = \exp \Big\{ -\int_{w_i}^{w_{i+1}}\dfrac{1}{D(w)}(\mathcal B[f](w,t)+D'(w))dw  \Big\},
\ee
for all $i=1,\dots,N-1$. 

Now, by equating the ratio $f_{i+1}/f_i$ of the numerical and the exact flux and setting
\be\label{eq:B_tilde}
\tilde{\mathcal C}_{i+1/2}=\dfrac{D_{i+1/2}}{\Delta w}\int_{w_i}^{w_{i+1}}\dfrac{\mathcal B[f](w,t)+D'(w)}{D(w)}dw
\ee
we recover
\be\label{eq:delta}
\delta_{i+1/2} = \dfrac{1}{\lambda_{i+1/2}}+\dfrac{1}{1-\exp(\lambda_{i+1/2})}, 
\ee
where
\be\label{eq:lambda_high}
\lambda_{i+1/2}=\int_{w_i}^{w_{i+1}}\dfrac{\mathcal B[f](w,t)+D'(w)}{D(w)}dw.
\ee
\begin{remark}
A second order method is obtained by discretizing \eqref{eq:lambda_high} through the midpoint rule
\be
\int_{w_i}^{w_{i+1}}\dfrac{\mathcal B[f](w,t)+D'(w)}{D(w)}dw \approx \dfrac{\Delta w(\mathcal B_{i+1/2}+D'_{i+1/2})}{D_{i+1/2}},
\ee
therefore
\be\label{eq:lambda}
\lambda_{i+1/2}^{\textrm{mid}} = \dfrac{\Delta w(\mathcal B_{i+1/2}+ D'_{i+1/2})}{D_{i+1/2}}
\ee
and 
\be
\delta_{i+1/2}^{\textrm{mid}} = \dfrac{D_{i+1/2}}{\Delta w(\mathcal B_{i+1/2}+ D'_{i+1/2})}++\dfrac{1}{1-\exp(\lambda_{i+1/2}^{\textrm{mid}})}.
\ee
Higher order accuracy of the steady state solution may be obtained by higher order approximations of the integral \eqref{eq:B_tilde}.
\end{remark}


\subsection{The multi-dimensional case}\label{sec:2D}
In order to extend the previous approach to multi-dimensional situations we consider here the case of two dimensional problems. We introduce a mesh consisting of the cells $C_{ij}=[w_{i-1/2},w_{i+1/2}]\times [v_{j-1/2},v_{j+1/2}]$ assumed to be of uniform size $\Delta w\Delta v$, where as usual $\Delta w:=w_{i+1/2}-w_{i-1/2}$ and $\Delta v:=v_{j+1/2}-v_{j-1/2}$ for all $i=0,\dots,N_1$ and $j=0,\dots,N_2$. 
 Integration of the general mean-field equation in dimension $d\ge1$ introduced in \eqref{eq:NAD_dim} yields 
\be\label{eq:eqflux2D}
\dfrac{d}{dt} f_{i,j} = \dfrac{\mathcal F_{i+1/2,j}[f]-\mathcal F_{i-1/2,j}[f]}{\Delta w}+\dfrac{\mathcal F_{i,j+1/2}[f]-\mathcal F_{i,j-1/2}[f]}{\Delta v},
\ee
being $\mathcal F_{i\pm1/2,j}[f]$, $\mathcal F_{i,j\pm1/2}[f]$ flux functions characterizing the numerical discretization. The quasi-stationary approximations over the cell $[w_i,w_{i+1}]\times[v_i,v_{i+1}]$ of the two dimensional problem read
\be\begin{split}
\int_{w_{i}}^{w_{i+1}} \dfrac{1}{f(w,v_j,t)}\partial_w f(w,v_j,t) dw&= -\int_{w_{i}}^{w_{i+1}}\dfrac{\mathcal B[f](w,v_j,t)+\partial_w D(w,v_j)}{D(w,v_j)}dw, \\
\int_{v_{j}}^{v_{j+1}} \dfrac{1}{f(w_i,v,t)}\partial_v f(w_i,v,t) dv&= -\int_{v_j}^{v_{j+1}}\dfrac{\mathcal B[f](w_i,v,t)+\partial_v D(w_i,v)}{D(w_i,v)}dv.
\end{split}\ee
Therefore setting
\be\begin{split}
\tilde{\mathcal C}_{i+1/2,j} &= \dfrac{D_{i+1/2,j}}{\Delta w}\int_{w_i}^{w_{i+1}}\dfrac{\mathcal B[f](w,v_j,t)+\partial_{w}D(w,v_j)}{D(w,v_j)}dw\\
\tilde{\mathcal C}_{i,j+1/2} &= \dfrac{D_{i,j+1/2}}{\Delta v}\int_{v_j}^{v_{j+1}}\dfrac{\mathcal B[f](w_i,v,t)+\partial_{v}D(w_i,v)}{D(w_i,v)}dv
\end{split}\ee
and by considering an analogous flux components by components as in the one-dimensional case
\be\begin{split}\label{eq:F2D}
\mathcal F_{i+1/2,j}[f] &= \tilde{\mathcal C}_{i+1/2,j}\tilde{f}_{i+1/2,j}+D_{i+1/2,j}\dfrac{f_{i+1,j}-f_{i,j}}{\Delta w}\\
\tilde f_{i+1/2,j}& = (1-\delta_{i+1/2,j})f_{i+1,j}+\delta_{i+1/2,j}f_{i,j}\\
\mathcal F_{i,j+1/2}[f] &= \tilde{\mathcal C}_{i,j+1/2}\tilde{f}_{i,j+1/2}+D_{i,j+1/2}\dfrac{f_{i,j+1}-f_{i,j}}{\Delta v}\\
\tilde f_{i,j+1/2}& = (1-\delta_{i,j+1/2})f_{i,j+1}+\delta_{i,j+1/2}f_{i,j},
\end{split}\ee
we define $\delta_{i+1/2,j}$ and $\delta_{i,j+1/2}$ in such a way that we preserve the steady state solution for each dimension, i.e.
\be\begin{split}\label{eq:delta2D_1}
\delta_{i+1/2,j}     &= \dfrac{1}{\lambda_{i+1/2,j}}+\dfrac{1}{1-\exp(\lambda_{i+1/2,j})},\\
\delta_{i,j+1/2}     &= \dfrac{1}{\lambda_{i,j+1/2}}+\dfrac{1}{1-\exp(\lambda_{i,j+1/2})}\\
\lambda_{i+1/2,j} &= \dfrac{\Delta w \tilde{\mathcal C}_{i+1/2,j}}{D_{i+1/2,j}},\quad \quad \lambda_{i,j+1/2} = \dfrac{\Delta v \tilde{\mathcal C}_{i,j+1/2}}{D_{i,j+1/2}}.
\end{split}\ee
The cases of higher dimension $d\ge 3$ may be derived in a similar way. 

\section{Main properties}
In order to study the structural properties of the numerical scheme, like non negativity and entropy property,  we restrict to the one-dimensional case. 

\subsection{Nonnegativity}
We introduce a time discretization $t^n=n\Delta t$ with $\Delta t>0$ and $n = 0,\dots,T$ and consider the simple forward Euler method
\be\label{eq:NAD_dimd}
f^{n+1}_i=f^n_i + \Delta t \dfrac{\mathcal F_{i+1/2}^n-\mathcal F_{i-1/2}^n}{\Delta w},
\ee
with no flux boundary conditions $F_{N+1/2}^n=\mathcal F_{-1/2}^n=0$.
 
\begin{lemma}
Let us consider the scheme (\ref{eq:NAD_dimd}) with no flux boundary conditions. We have for all $n\in\mathbb N$
\be
\sum_{i=0}^N f^{n+1}_i = \sum_{i=0}^N f^n_i.
\ee
\end{lemma}
\begin{proof}
From equation (\ref{eq:NAD_dimd}) we have
\be
\sum_{i=0}^N f^{n+1}_i = \sum_{i=0}^N f^n_i +\dfrac{\Delta t}{\Delta w} \sum_{i=0}^N (\mathcal F_{i+1/2}^n-\mathcal F_{i-1/2}^n).
\ee
Now since
\be
\sum_{i=0}^N (\mathcal F_{i+1/2}^n-\mathcal F_{i-1/2}^n) = \mathcal F_{N+1/2}^n-\mathcal F_{-1/2}^n,
\ee
by imposing no flux boundary conditions we conclude. 

\end{proof}
Note that mass conservation holds true also in the backward Euler case by imposing $\mathcal F_{N+1/2}^{n+1}=\mathcal F_{-1/2}^{n+1}=0$.

Concerning non negativity we can prove \cite{PZ}
\begin{proposition}
Under the time step restriction 
\be
\Delta t\le \dfrac{\Delta w^2}{2(M\Delta w+D)},\quad
M = \max_{0\le i\le N} |\tilde{\mathcal C}_{i+1/2}^n|, \quad D = \max_{0\le i\le N}D_{i+1/2},
\label{eq:nu}
\ee
the explicit scheme (\ref{eq:NAD_dimd}) preserves nonnegativity, i.e 
$ f^{n+1}_i\ge 0$ if $f^n_i\ge 0$, $i=0,\dots,N$. 
\end{proposition}
\begin{proof}
The scheme reads
\be\begin{split}\label{eq:prop1_f}
& f_i^{n+1} = f_i^n + \dfrac{\Delta t}{\Delta w}\Bigg[ \left( (1-\delta_{i+1/2}^n)\tilde{\mathcal C}_{i+1/2}^{n}+\dfrac{D_{i+1/2}}{\Delta w} \right)f_{i+1}^n\\
& +\left(\tilde{\mathcal C}_{i+1/2}^n\delta_{i+1/2}^n-\tilde{\mathcal C}_{i-1/2}^n(1-\delta_{i-1/2}^n)-\dfrac{1}{\Delta w}(D_{i+1/2}+D_{i-1/2})\right)f_i^n \\
&- \left(\tilde{\mathcal C}^n_{i-1/2}\delta_{i-1/2}^n-\dfrac{D_{i-1/2}}{\Delta w}\right)f_{i-1}^n \Bigg].
\end{split}\ee
From \eqref{eq:prop1_f} the coefficients of $f_{i+1}^n$ and $f_{i-1}^n$ should satisfy
\be\begin{split}
(1-\delta_{i+1/2})\tilde{\mathcal C}_{i+1/2}^n+\dfrac{D_{i+1/2}}{\Delta w}\ge 0, \qquad
-\delta_{i-1/2}\tilde{\mathcal C}_{i-1/2}^n+\dfrac{D_{i-1/2}}{\Delta w}\ge 0,
\end{split}\ee
that is equivalent to show that
\be\begin{split}
\lambda_{i+1/2}\left(1-\dfrac{1}{1-\exp{\lambda_{i+1/2}}}\right)\ge 0, \qquad \dfrac{\lambda_{i-1/2}}{\exp{\lambda_{i-1/2}}-1}\ge 0,
\end{split}\ee
which holds true thanks to the properties of the exponential function. In order to ensure the non negativity of the scheme the time step should satisfy the restriction $\Delta t\le {\Delta w}/{\nu}$, with 
\be
\nu = \max_{0\le i\le N} \Big\{ \tilde{\mathcal C}_{i+1/2}^n\delta^n_{i+1/2}-\tilde{\mathcal C}_{i-1/2}^n(1-\delta_{i-1/2}^n)-\dfrac{D_{i+1/2}+D_{i-1/2}}{\Delta w} \Big\}.
\ee
Being $M$ defined in \eqref{eq:nu}, and $0\le \delta_{i\pm 1/2}\le 1$, we obtain the prescribed bound. 
\end{proof}
\begin{remark}
Higher order SSP methods \cite{GST} are obtained by considering a convex combination of forward Euler methods. Therefore, the non negativity result can be extended to general SSP methods. 
\end{remark}
In practical applications, it is desirable to avoid the parabolic restriction $\Delta t = O((\Delta w)^2)$ of explicit schemes. Unfortunately fully implicit methods originate a nonlinear system of equations. However, we can prove that nonnegativity of the solution holds true also for the semi-implicit case
\be
f^{n+1}_i=f^n_i + \Delta t \dfrac{\hat{\mathcal F}_{i+1/2}^{n+1}-\hat{\mathcal F}_{i-1/2}^{n+1}}{\Delta w},
\label{eq:semi}
\ee
where
\be
\hat{\mathcal F}_{i+1/2}^{n+1}= \tilde{\mathcal C}_{i+1/2}^n \left[ (1-\delta_{i+1/2}^n)f_{i+1}^{n+1}+\delta_{i+1/2}f_i^{n+1} \right]+D_{i+1/2}\dfrac{f_{i+1}^{n+1}-f_i^{n+1}}{\Delta w}.
\ee
We have \cite{PZ}
\begin{proposition}
Under the time step restriction 
\be\label{eq:time_step_implicit}
\Delta t< \dfrac{\Delta w}{2M},\qquad M = \max_{0\le i\le N}|\tilde{\mathcal B}^n_{i+1/2}|
\ee
the semi-implicit scheme (\ref{eq:semi}) preserves nonnegativity, i.e 
$ f^{n+1}_i\ge 0$ if $f^n_i\ge 0$, $i=0,\dots,N$. 
\end{proposition}
\begin{proof}
Setting $\alpha_{i+1/2}^{n}=\dfrac{\lambda_{i+1/2}^{n}}{\exp(\lambda_{i+1/2}^{n})-1}$ and 
\be\begin{split}\label{eq:RQP}
& R_i^{n} = 1+\dfrac{\Delta t}{\Delta w^2}\left[ D_{i+1/2}\alpha_{i+1/2}^{n}+D_{i-1/2}\alpha_{i-1/2}^{n}\exp(\lambda_{i-1/2}^{n})  \right] \\
& Q_i^{n} = \dfrac{\Delta t}{\Delta w^2} D_{i+1/2}\alpha_{i+1/2}^{n}\exp(\lambda_{i+1/2}^{n})\\
& P_i^{n} = \dfrac{\Delta t}{\Delta w^2}D_{i-1/2}\alpha_{i-1/2}^{n},
\end{split}\ee
equation \eqref{eq:semi} corresponds to 
\be\label{eq:implicit_2}
R_i^{n} f_i^{n+1} - Q_i^{n}f_{i+1}^{n+1}-P_i^{n}f_{i-1}^{n+1}=f_i^n.
\ee
If we introduce the matrix 
\be
(\mathcal A[f^{n}])_{ij}=
\begin{cases}
R_i^{n}, & j=i\\
-Q_i^{n}, & j = i+1, 1\le i\le N \\
-P_i^{n}, & j = i-1, 0\le i \le N-1,
\end{cases}
\ee
with $R_i^{n}>0$, $Q_i^{n}>0$, $P_i^{n}>0$ defined in \eqref{eq:RQP} the semi-implicit scheme may be expressed in matrix form as follows
\be\label{eq:iter_S}
\mathcal A[\textbf{f}^{n}] \textbf{f}^{n+1}=\textbf{f}^n,
\ee
with $\textbf{f}^{n}=\left(f_0^n,\dots,f_N^n\right)$. Now the matrix $\mathcal A$ is strictly diagonally dominant if and only if 
\be
|R_i^n|>|Q_i^n|+|P_i^n|,\qquad i=0,1\dots,N,
\ee
condition which holds true if
\be\begin{split}\label{eq:strong_diagonal}
1&> \dfrac{\Delta t}{\Delta w^2}\left[ D_{i+1/2}\alpha_{i+1/2}^{n}\left(\exp(\lambda_{i+1/2}^{n})-1\right)-D_{i-1/2}\alpha_{i-1/2}^{n}\left( \exp(\lambda_{i-1/2}^{n})-1 \right) \right]\\
&=\dfrac{\Delta t}{\Delta w^2}\left[ D_{i+1/2}\lambda_{i+1/2}^{n}-D_{i-1/2}\lambda_{i-1/2}^{n} \right]
= \dfrac{\Delta t}{\Delta w}\left[\tilde{\mathcal B}_{i+1/2}^{n}-\tilde{\mathcal B}_{i-1/2}^{n} \right].
\end{split}\ee
\end{proof}




\subsection{Entropy property}
In order to discuss the entropy property we consider the prototype equation \cite{FPTT,PZ}
\be\label{eq:wu}
\partial_t f(w,t) = \partial_w \left[ (w-u)f(w,t) + \partial_w (D(w)f(w,t)) \right], \qquad w\in  I=[-1,1],
\ee
with $-1<u<1$ a given constant and boundary conditions
\be\label{eq:wu_boundary}
\partial_w (D(w)f(w,t))+(w-u)f(w,t) = 0, \qquad w=\pm1.
\ee
If the stationary state $f^\infty$ exists equation \eqref{eq:wu} may be written in the form
\be\label{eq:nonlog_landau}
\partial_t f(w,t) = \partial_w \left[ D(w)f^{\infty}(w)\partial_w \left(\dfrac{f(w,t)}{f^{\infty}(w)}\right) \right].
\ee
We define the relative entropy for all positive functions $f(w,t),g(w,t)$ as follows
\be\label{eq:rel_entropy}
\mathcal H(f,g) = \int_I f(w,t) \log\left(\dfrac{f(w,t)}{g(w,t)} \right),
\ee
we have \cite{FPTT}
\be
\dfrac{d}{dt}\mathcal H(f,f^{\infty}) = -\mathcal I_D(f,f^{\infty}),
\ee
where the dissipation functional $\mathcal I_D(\cdot,\cdot)$ is defined as 
\be\begin{split}
\mathcal I_D(f,f^{\infty})& = 
 \int_{\mathcal I} D(w)f(w,t)\left(\partial_w \log \left( \dfrac{f(w,t)}{f^{\infty}(w)}\right)\right)^2dw,\\
& = 
 \int_{\mathcal I} D(w)f^{\infty}(w,t)\partial_w \log \left( \dfrac{f(w,t)}{f^{\infty}(w)}\right)\partial_w\left(\frac{f}{f^{\infty}}\right)dw.\\ 
\end{split}\ee

\begin{lemma}\label{lem:flux_CCE}
 In the case $\mathcal B[f](w,t)=\mathcal B(w)$ the numerical flux function (\ref{eq:CC_flux})-(\ref{eq:f_CC}) with $\tilde{\mathcal B}_{i+1/2}$ and $\delta_{i+1/2}$ given by (\ref{eq:B_tilde})-(\ref{eq:delta}) can be written in the form (\ref{eq:nonlog_landau}) and reads
 \be
 \mathcal F_{i+1/2} = \dfrac{D_{i+1/2}}{\Delta w} \hat f_{i+1/2}^{\infty} \left( \dfrac{f_{i+1}}{f^{\infty}_{i+1}}-\dfrac{f_i}{f^{\infty}_i} \right),
 \label{eq:nnll}
 \ee
 with 
 \be
 \hat{f}^{\infty}_{i+1/2} = \dfrac{f_{i+1}^{\infty}f_i^{\infty}}{f_{i+1}^{\infty}-f_i^{\infty}}\log \left(\dfrac{f_{i+1}^{\infty}}{f_i^{\infty}}\right).
 \ee
 \end{lemma}
 \begin{proof}
 In the hypothesis $\mathcal B[f](w,t)=\mathcal  B(w)$ the definition of $\lambda_{i+1/2}$ does not depends on time, i.e. $\lambda_{i+1/2}=\lambda_{i+1/2}^{\infty}$ and if a steady state exists we may write
 \be
 \log f_{i}^{\infty}-\log f_{i+1}^{\infty} = \lambda_{i+1/2}.
 \ee
Furthermore, the flux function $\mathcal F_{i+1/2}$ assumes the following form
 \be\begin{split}
 \mathcal F_{i+1/2} &= \dfrac{D_{i+1/2}}{\Delta w}\left[ \lambda_{i+1/2}\tilde{f}_{i+1/2} +(f_{i+1}-f_i)\right]\\
 &=  \dfrac{D_{i+1/2}}{\Delta w}\left[ \lambda_{i+1/2} (f_{i+1}+\delta_{i+1/2}(f_i-f_{i+1})) +(f_{i+1}-f_i)\right],
 \end{split}\ee
 where 
 \be
 \delta_{i+1/2} = \dfrac{1}{\log f_{i}^{\infty}-\log f_{i+1}^{\infty}}+\dfrac{f_{i+1}^{\infty}}{f_{i+1}^{\infty}-f_i^{\infty}}.
 \ee
 Hence we have
  \be\begin{split}
 \mathcal F_{i+1/2}^n = \dfrac{D_{i+1/2}}{\Delta w} \log \left( \dfrac{f^{\infty}_i}{f_{i+1}^{\infty}}\right)  &\left[f_{i+1}+ \left(\dfrac{f_i-f_{i+1}}{\log f_{i}^{\infty}-\log f_{i+1}^{\infty}}+\dfrac{f_{i+1}^{\infty}(f_i-f_{i+1})}{f_{i+1}^{\infty}-f_i^{\infty}}\right)\right.\\
 &\left.\quad+\dfrac{f_{i+1}-f_{i}}{\log f_i^{\infty}-\log f_{i+1}^{\infty}}  \right], \\
 = \dfrac{D_{i+1/2}}{\Delta w} \log\left( \dfrac{f^{\infty}_i}{f_{i+1}^{\infty}}\right)& \left( \dfrac{f_{i+1}^{\infty}f_i-f_{i}^{\infty}f_{i+1}}{f_{i+1}^{\infty}-f_i^{\infty}} \right)
 \end{split}\ee
  which gives (\ref{eq:nnll}).
 \end{proof}

%
%

\begin{theorem}\label{th:1}
Let us consider $\mathcal B[f](w,t)=w-u$ as in equation \eqref{eq:wu}. The numerical flux (\ref{eq:CC_flux})-(\ref{eq:f_CC}) with $\tilde{\mathcal B}_{i+1/2}$ and $\delta_{i+1/2}$ given by (\ref{eq:B_tilde})-(\ref{eq:delta}) satisfies the discrete entropy dissipation
\be
\dfrac{d}{dt}\mathcal H_{\Delta}(f,f^{\infty})=-\mathcal I_{\Delta}(f,f^{\infty}),
\ee
where
\be
\mathcal H_{\Delta w}(f,f^{\infty}) = \Delta w \sum_{i=0}^N f_i \log \left(\dfrac{f_i}{f_i^{\infty}} \right)
\ee
and $I_{\Delta}$ is the positive discrete dissipation function 
\be
\mathcal I_{\Delta}(f,f^{\infty}) = \sum_{i=0}^N
 \left[ \log \left(\dfrac{f_{i+1}}{f^{\infty}_{i+1}}\right)-\log\left(\dfrac{f_i}{f_i^{\infty}}\right) \right]\cdot \left(\dfrac{f_{i+1}}{f_{i+1}^{\infty}}-\dfrac{f_i}{f_{i}^{\infty}}\right)\hat{f}_{i+1/2}^{\infty}D_{i+1/2}\ge 0.
\ee
\end{theorem}
\begin{proof}
From the definition of relative entropy we have
\be\begin{split}
\dfrac{d}{dt} \mathcal H(f,f^{\infty})&= \Delta w\sum_{i=0}^N  \dfrac{df_i}{dt}\left(\log\left(\dfrac{f_i}{f_i^{\infty}}\right)+1\right)\\
&= \Delta w \sum_{i=0}^N \left( \log\left(\dfrac{f_i}{f_i^{\infty}}\right)+1 \right)(\mathcal F_{i+1/2}-\mathcal F_{i-1/2}),
\end{split}\ee
and after summation by parts we get
\be
\dfrac{d}{dt} \mathcal H(f,f^{\infty})=-\Delta w \sum_{i=0}^N \left[ \log \left(\dfrac{f_{i+1}}{f_{i+1}^{\infty}}\right)-\log\left(\dfrac{f_i}{f_i^{\infty}}\right) \right]\mathcal F_{i+1/2}.
\ee
Thanks to the identity of Lemma \ref{lem:flux_CCE}  we may conclude since the function $(x-y)\log(x/y)$ is non-negative for all $x,y\ge 0$.
\end{proof}

\section{Numerics}
In this section we present several numerical tests for the proposed structure--preserving schemes. In particular, we show that the schemes accurately describe the steady state solution of mean-field equations. 

\subsubsection*{Test 1: accuracy and steady states} Let us consider the evolution of a distribution described by the equation (\ref{eq:wu}) with 
\be
u= \int_I v f(v,t)dv, \qquad D(w) = \dfrac{\sigma^2}{2}(1-w^2)^2.
\ee
We consider as initial distribution  
\be\label{eq:opinion_initial}
f(w,0) = \beta \left[ \exp\{-c(w+1/2)\}+\exp\{-c(w-1/2)\} \right],\qquad c = 30,
\ee
and $\beta>0$ a normalization constant. The stationary solution in this case can be explicitly computed and is given by (\ref{eq:ops}).

We compute the relative $L^1$ error of the solution with respect to the stationary state using $N=41$ points. In Figure \ref{fig:opinion_1} we show the evolution of the mean--field equation and the relative $L^1$-error in approximating the steady state solution. We used open Newton-Cotes formulas of various orders and Gaussian quadrature to evaluate (\ref{eq:lambda_high}). It is possible to observe how the different integration methods capture the steady state with different accuracy. In particular using Gaussian quadrature we essentially reached machine precision.  

\begin{figure}[t]\label{fig:opinion_1}
\centering
\subfigure[]{
\includegraphics[scale=0.29]{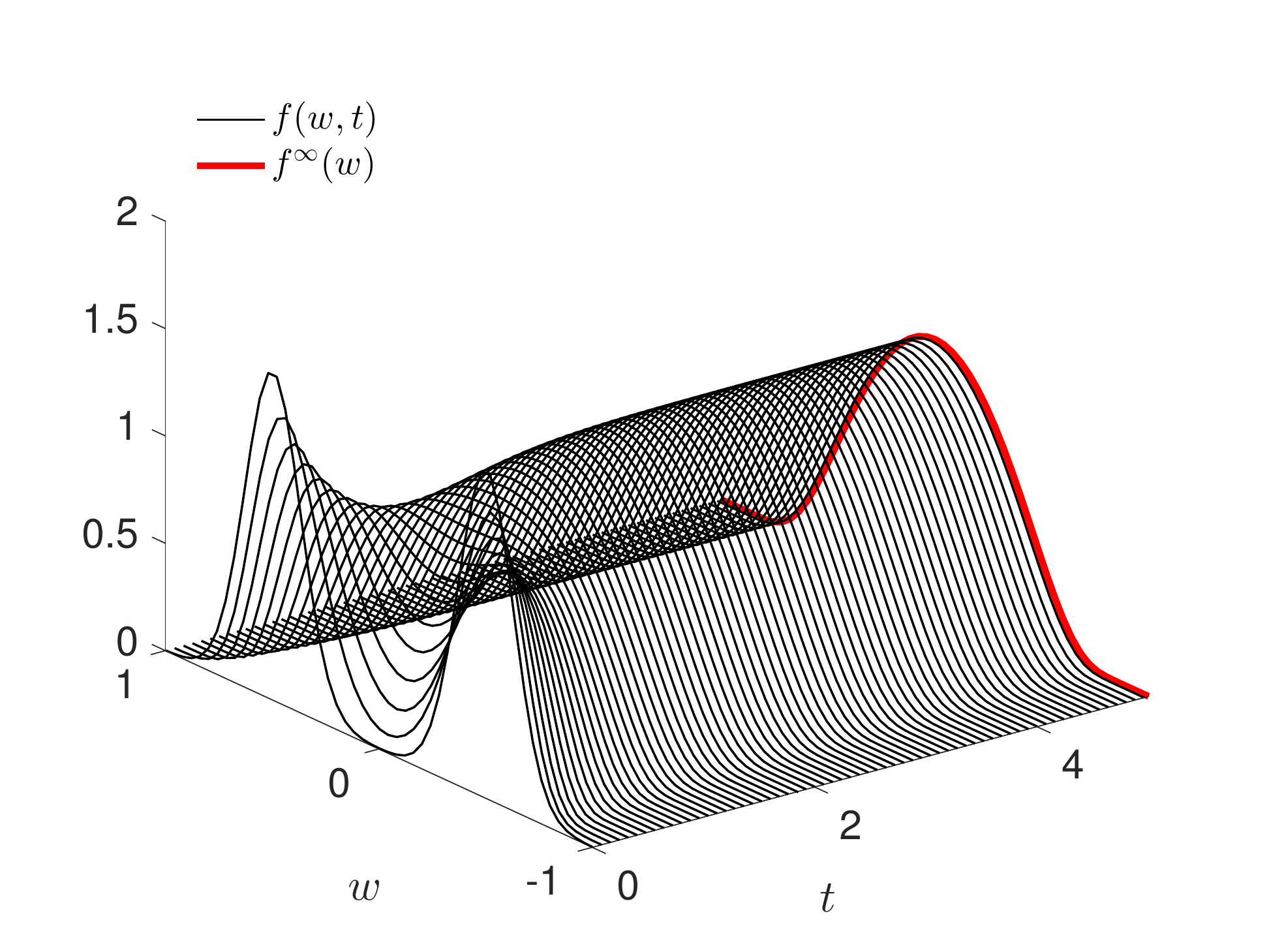}}
\subfigure[]{\includegraphics[scale=0.29]{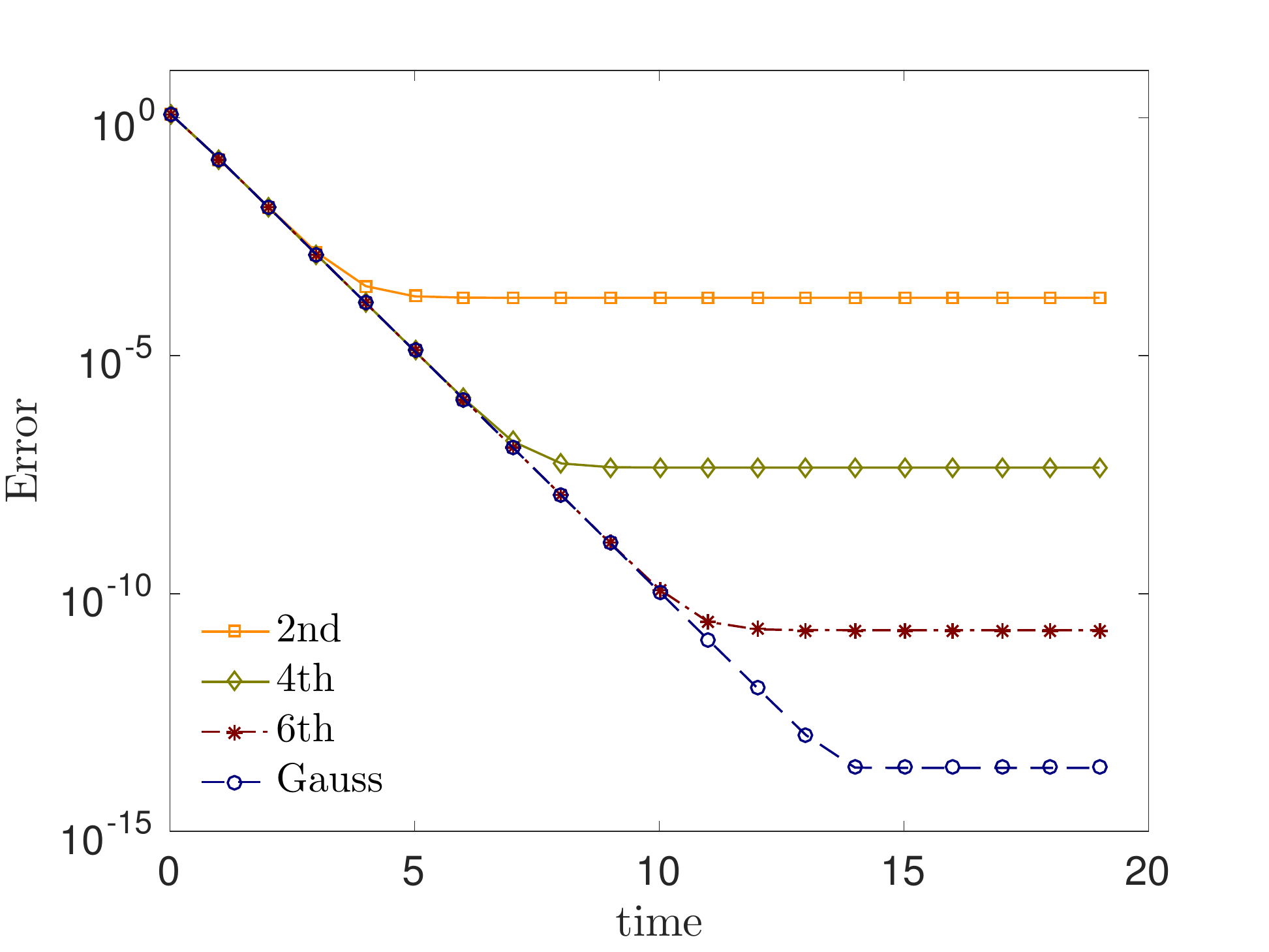}}
\caption{Test 1. \textbf{(a)} Time evolution of the density $f(w,t)$ for problem (\ref{eq:wu}) with initial datum \eqref{eq:opinion_initial} over the time interval $[0,5]$ for $\sigma^2/2=0.1$, $\Delta w = 0.05$. \textbf{(b)} Evolution of the relative $L^1$ error with respect to the stationary solution (\ref{eq:ops}) for various quadrature methods. }
\end{figure}

In Table \ref{tab:2} we estimate the overall order of convergence of the scheme for various integration methods. Here we used $N=41,81,161$ grid points. The time integration has been performed with an explicit RK4 method and the time step is chosen in such a way that the CFL condition for the positivity of the scheme is satisfied, therefore $\Delta t=O((\Delta w)^2)$. As expected the methods are second order accurate in transient regimes and, as they approach the steady state, they reach the order of the quadrature method. Clearly, the order of Gaussian quadrature is bounded by the maximum observable order which is $8$ due to the choice of the time discretization method.

\begin{table}
\begin{center}
\begin{tabular}{ |c || c | c | c | c |  }
\hline
                                      & $2$nd  &  $4$th & $6$th & Gauss \\ \hline
\multirow{2}{*} {$T=1$} & 1.8676 &  1.9972 & 1.9958 & 1.9958  \\
& 1.9840 & 1.9991 & 1.9987 & 1.9987 \\ 
\multirow{2}{*}{$T=5$} & 1.9348 & 3.2518 & 2.3578 & 2.3344 \\
& 2.0043 & 2.6218 & 2.0948 & 2.0930 \\ 
\multirow{2}{*}{$T=10$} & 1.9289 & 3.9178 & 6.4645 & 7.3482 \\
& 2.0034 & 3.9185 & 6.3630 & 7.9217 \\
\multirow{2}{*}{$T=15$} & 1.9289 & 3.9178 & 6.4701 & 7.3512 \\
& 2.0034 & 3.9786 & 6.6021 & 7.9954 \\ 
\hline
\end{tabular}\label{tab:2}
\end{center}
\caption{Test 1. Estimation of the order of convergence toward the reference stationary state for each integration method at different times.}
\end{table}

\subsubsection*{Test 2: flocking dynamics} 
We consider a mean-field Cucker-Smale flocking model as introduced in (\ref{eq:CS}). The space variable is discretized using a third order WENO scheme, and the transport and interaction process are combined using a second order Strang splitting scheme. 
For the mean-field term, we considered a semi--implicit scheme with Gaussian quadrature of the weights. This choice guarantees spectral accuracy for the description of the steady state solution of the equation. 

In Figure \ref{fig:CS} we report the evolution of the solution $f(x,w,t)$ in the phase space $(x,w)\in[-3,3]\times[-5,5]$ with $\Delta x = 6\cdot 10^{-2}$ and $\Delta w = 5\cdot 10^{-2}$. The time step has been chosen in order to satisfy the CFL condition $\Delta t/\Delta x=0.25/\max(w)$.

We considered $\gamma=0.1<1/2$ in the Cucker--Smale interaction function (\ref{eq:csH}) and a constant diffusion $D(x,w)=0.1$. The initial datum is here given by a multivariate population which shares the same average space location $x=0$ and is strongly clustered around opposite velocities $v=\pm 1.5$. As expected, the whole system converge to the same velocity, i.e. the distribution tends to concentrate in the velocity space and to be distributed uniformly along the spatial dimension. 

\begin{figure}[t]\label{fig:CS}
\centering
\subfigure[$t=0$]{\includegraphics[scale=0.213]{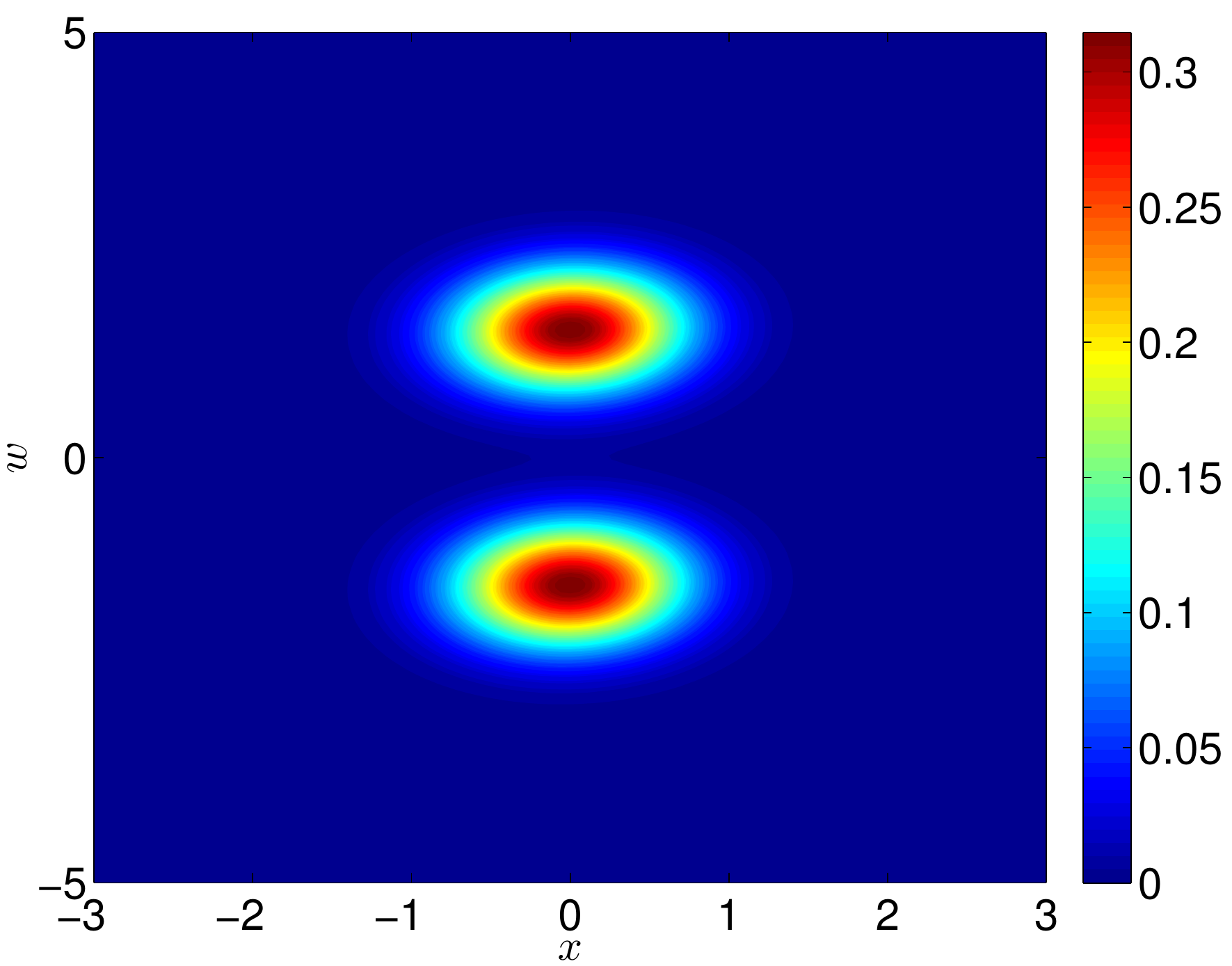}}
\subfigure[$t=0.6$]{\includegraphics[scale=0.213]{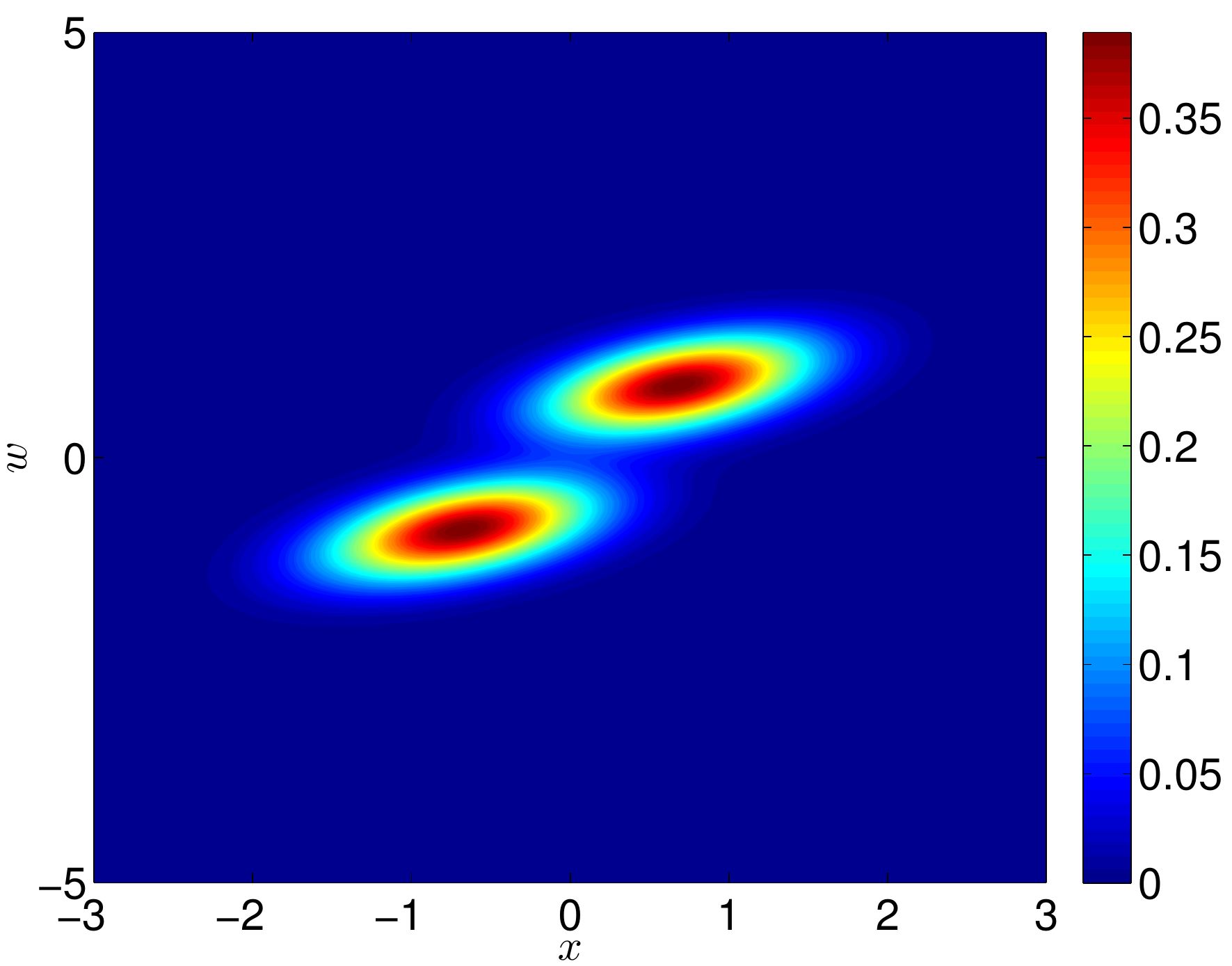}}
\subfigure[$t=1.2$]{\includegraphics[scale=0.213]{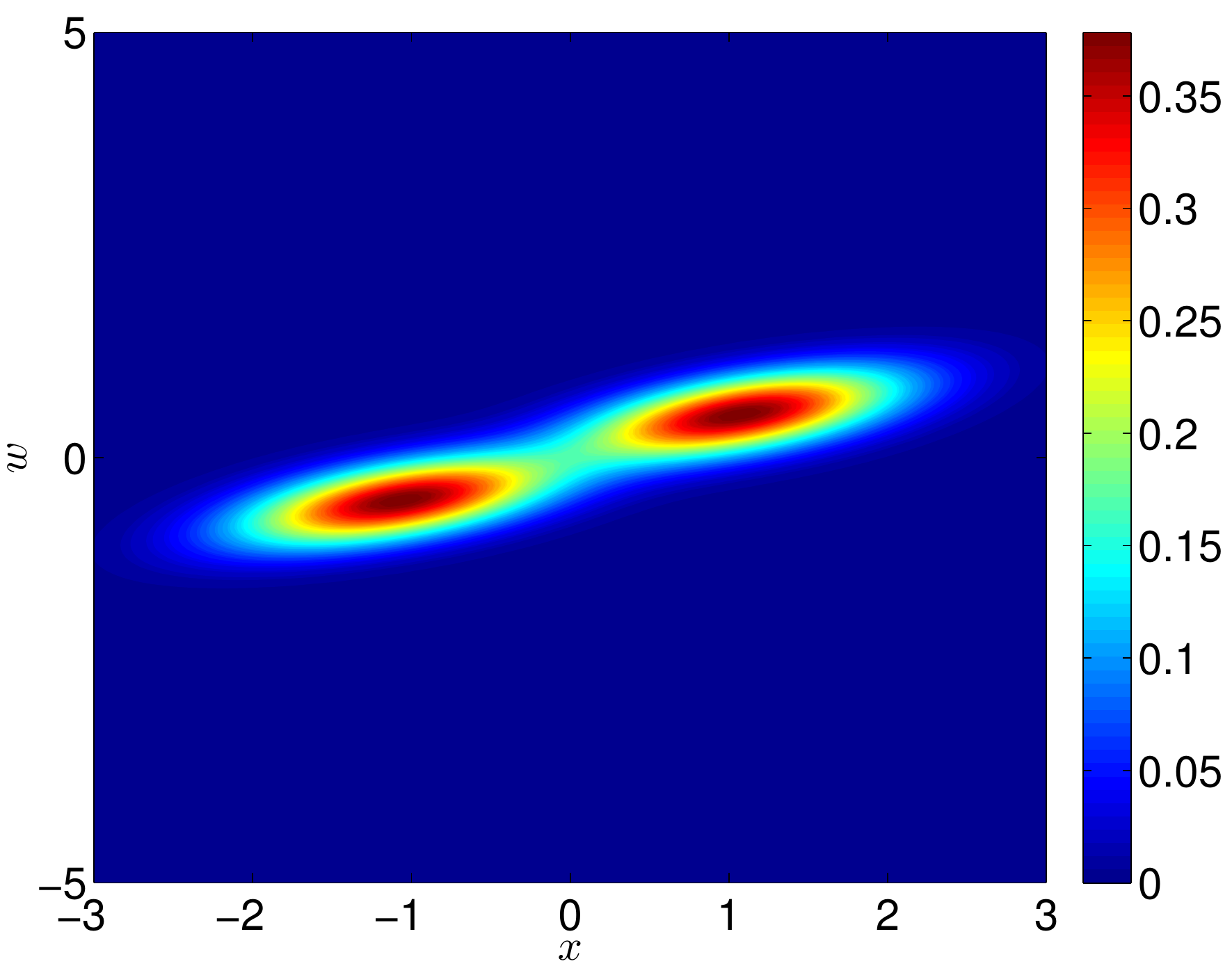}}\\
\subfigure[$t=3.0$]{\includegraphics[scale=0.213]{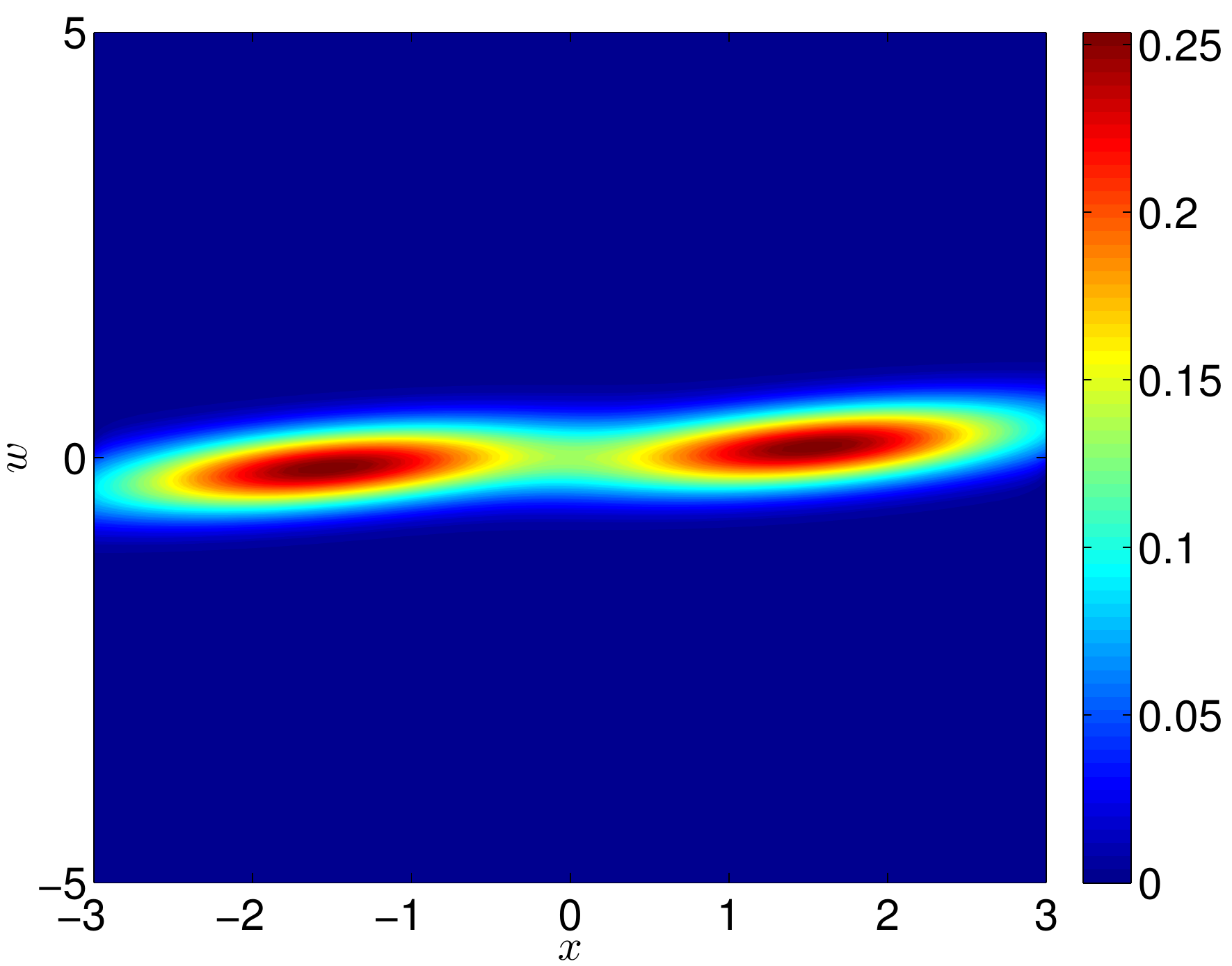}}
\subfigure[$t=6.0$]{\includegraphics[scale=0.213]{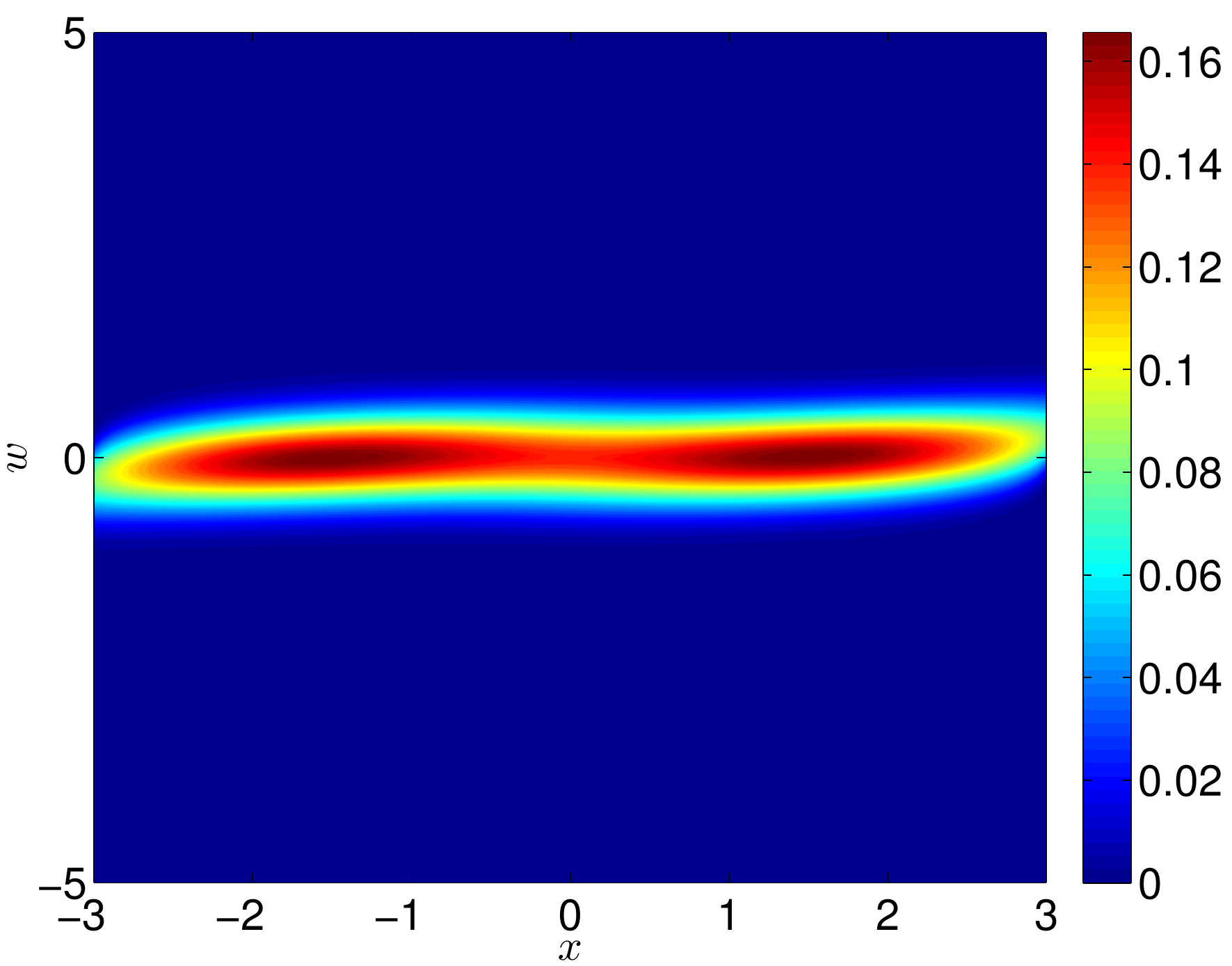}}
\subfigure[$t=9.0$]{\includegraphics[scale=0.213]{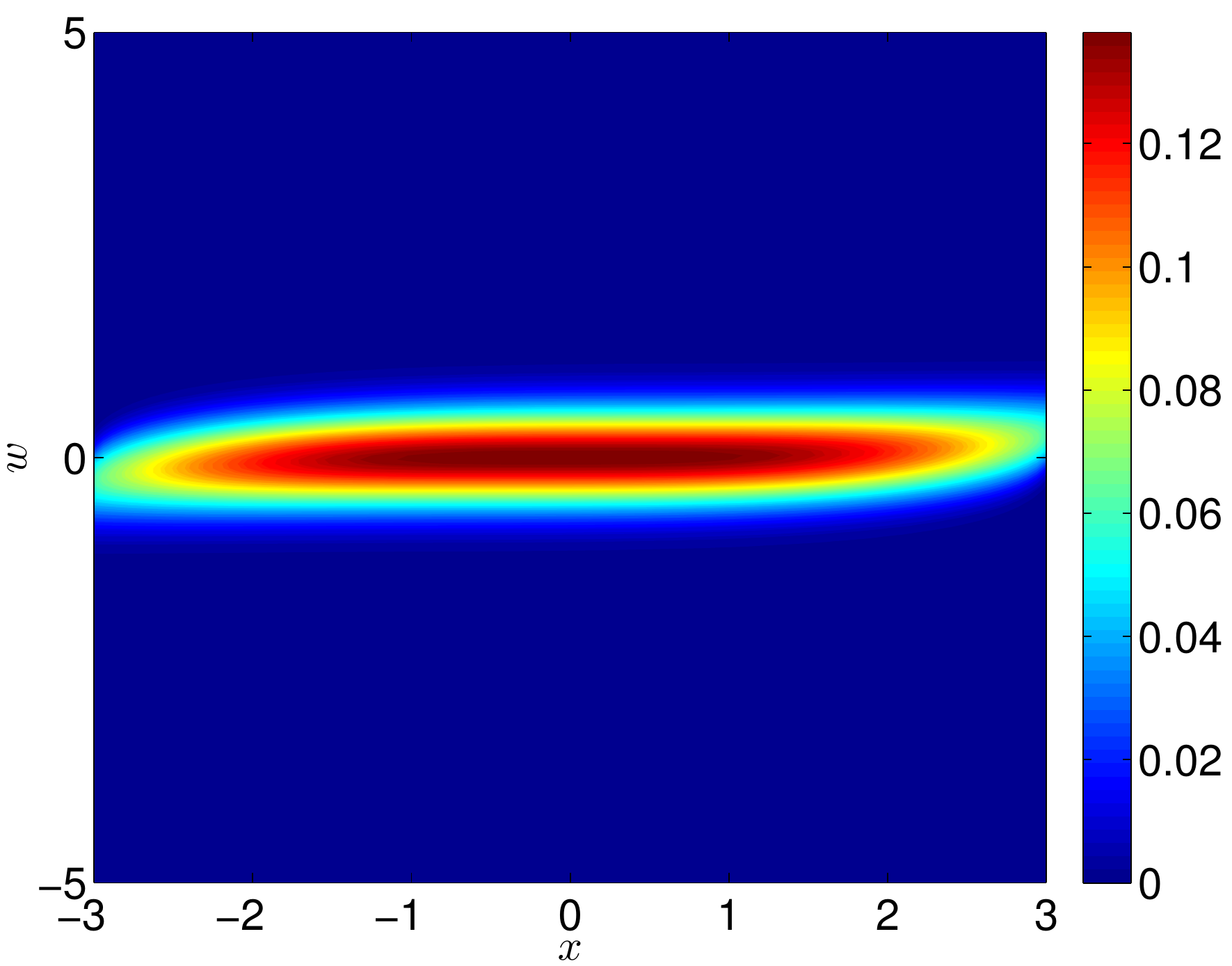}}
\caption{Test 2. Mean--field Cucker--Smale model for $(x,w)\in[-3,3]\times[-5,5]$ with $\Delta x = 6\cdot 10^{-2}$ and $\Delta w = 5\cdot 10^{-2}$, $\Delta t/\Delta x=0.25/\max(w)$. We considered $\gamma=0.1$ in (\ref{eq:csH}) and a constant diffusion function $D=0.1$.  }
\end{figure}

\subsubsection*{Test 3: opinion on networks}
Finally, we consider the model of opinion on networks (\ref{eq:APZ}). We focus on the case of a connection dependent bounded confidence model, where the agents interact only within a certain range of confidence. Hence, we define the compromise function \cite{APZ3}
\be\label{eq:HK}
P(w,v;x,y) = \chi_{\{|w-v|\le \Delta(x)\}}(v),
\ee
where $\Delta (x) = d_0 \dfrac{x}{c_{\textrm{max}}}$ and $D(w,x)=(1-w^2)^2$. This choice reflects a behavior where agents with higher number of connections are prone to larger level of confidence.  We report in Figure \ref{fig:HK} the evolution of the solution (where in order to better show its evolution we plotted $\log(f(w,x,t)+\epsilon)$, with $\epsilon = 0.001$). We can observe how the introduction of the function $\Delta(c)$ creates a heterogeneous emergence of clusters with respect to the connectivity level: for higher level of connectivity consensus is reached, since the bounded confidence level is larger, instead for lower levels of connectivity multiple clusters appears. In the limiting case $c=0$ the opinions are not influenced by the consensus dynamics.

\begin{figure}[t]\label{fig:HK}
\centering
\subfigure[$t=0$]{\includegraphics[scale= 0.26]{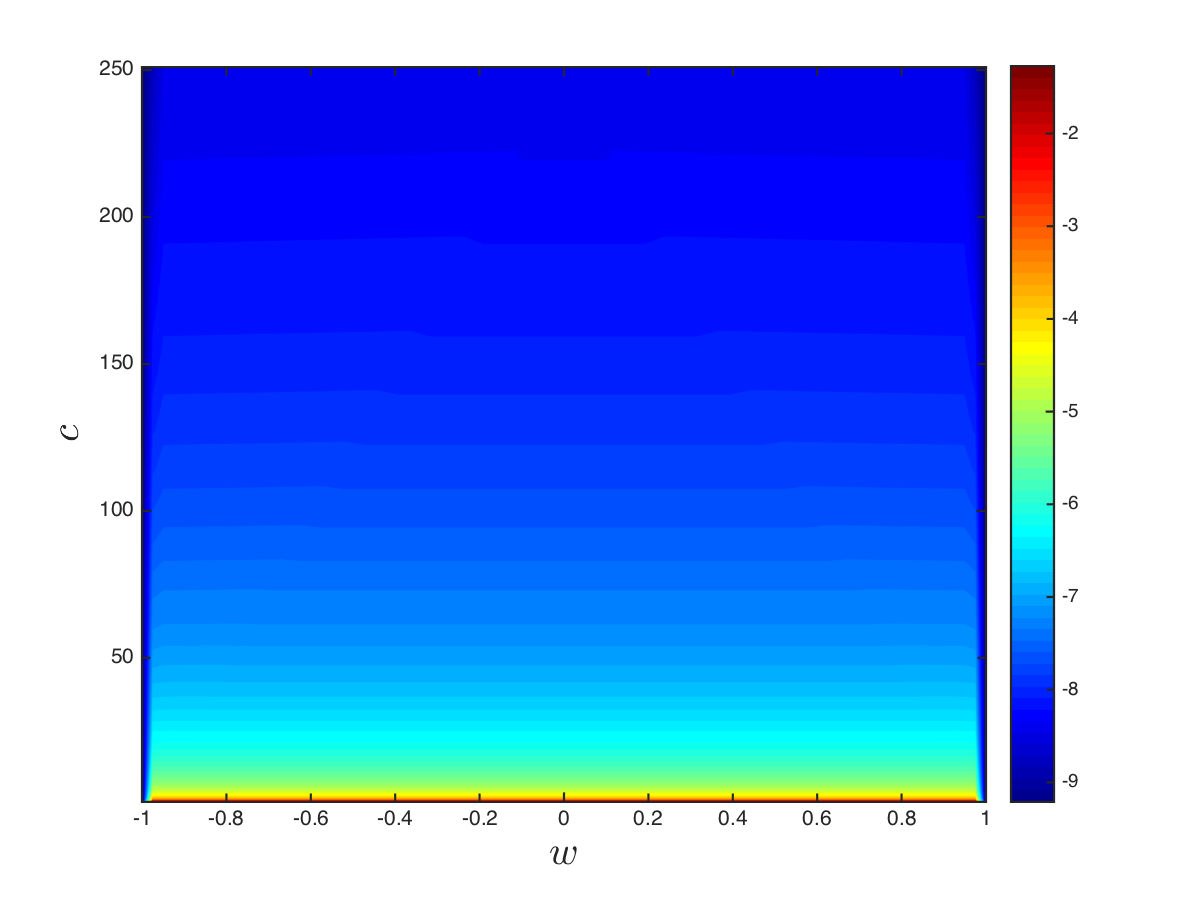}}
\subfigure[$t=10$]{\includegraphics[scale= 0.26]{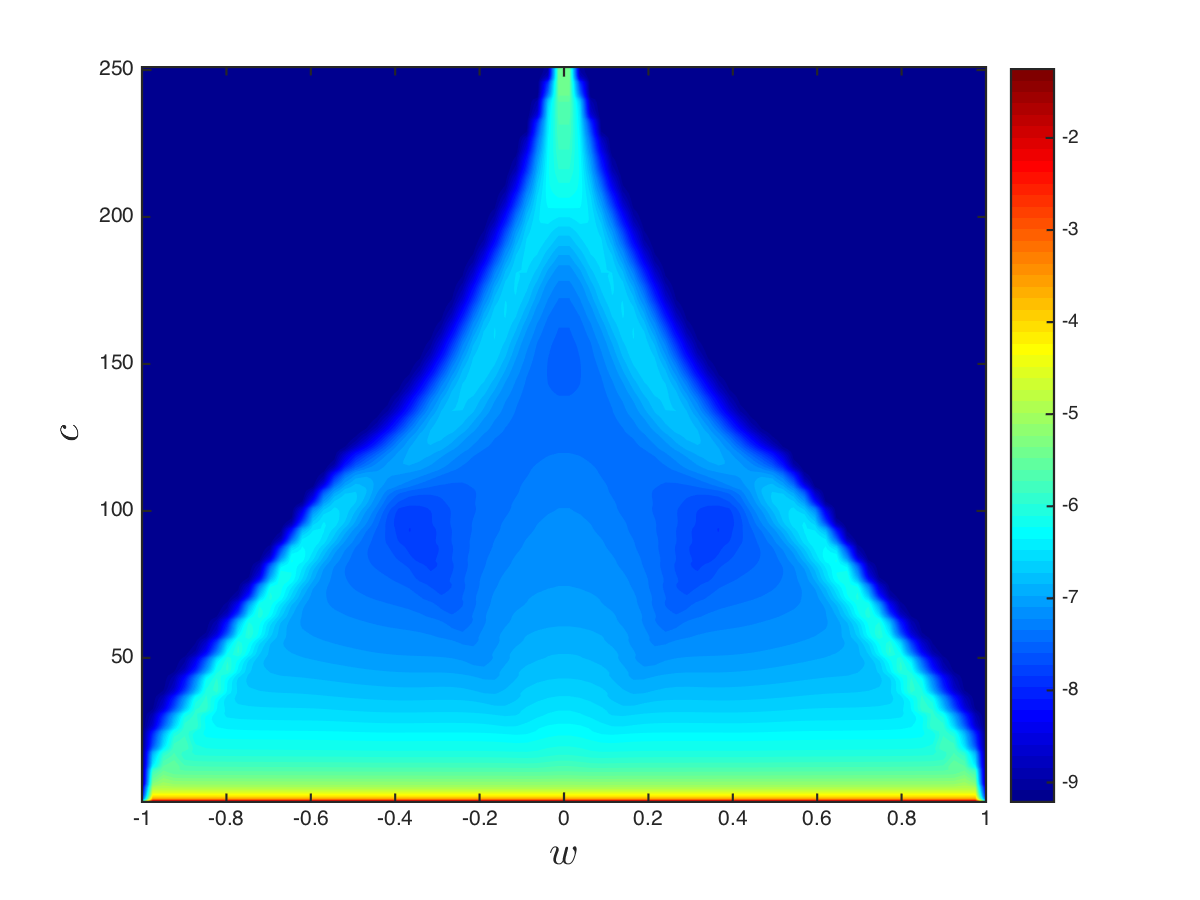}}
\subfigure[$t=50$]{\includegraphics[scale= 0.26]{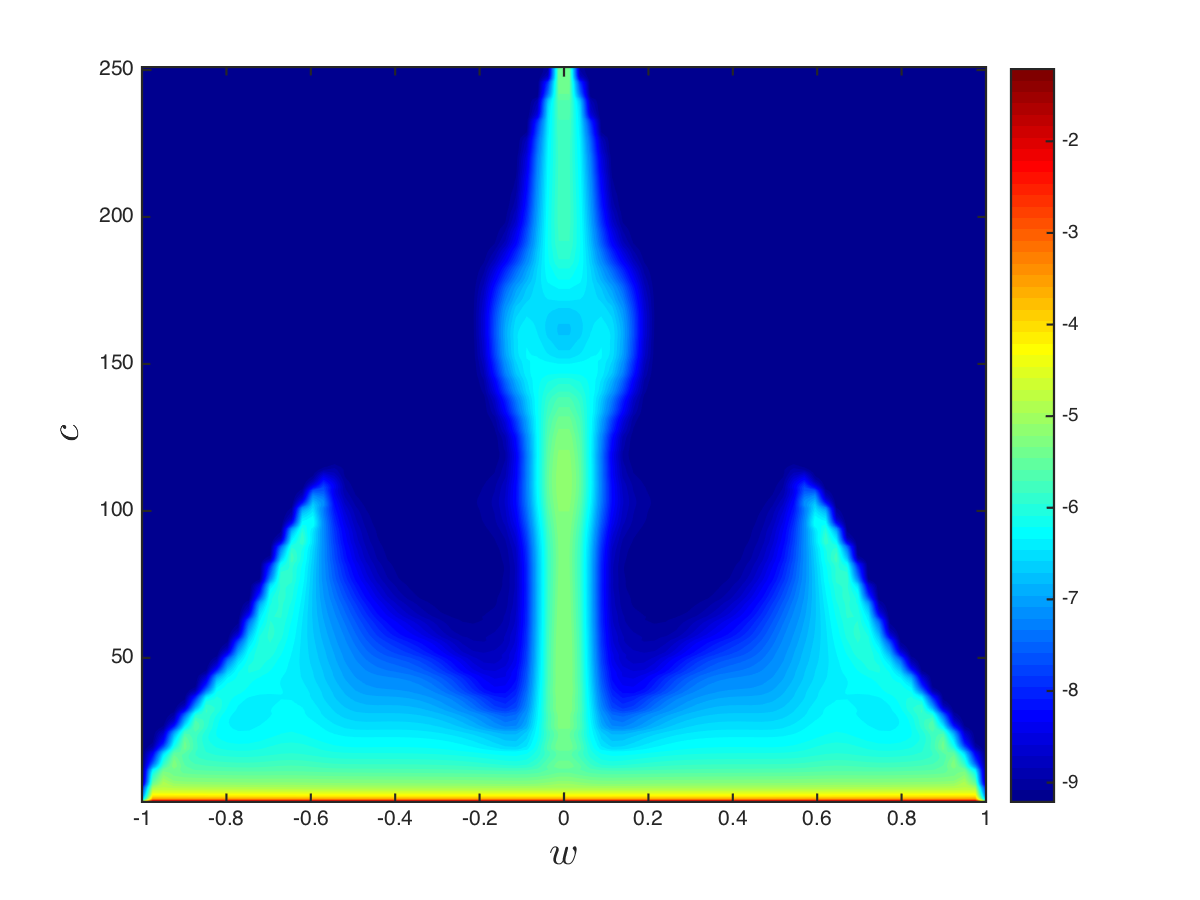}}
\subfigure[t=100]{\includegraphics[scale= 0.26]{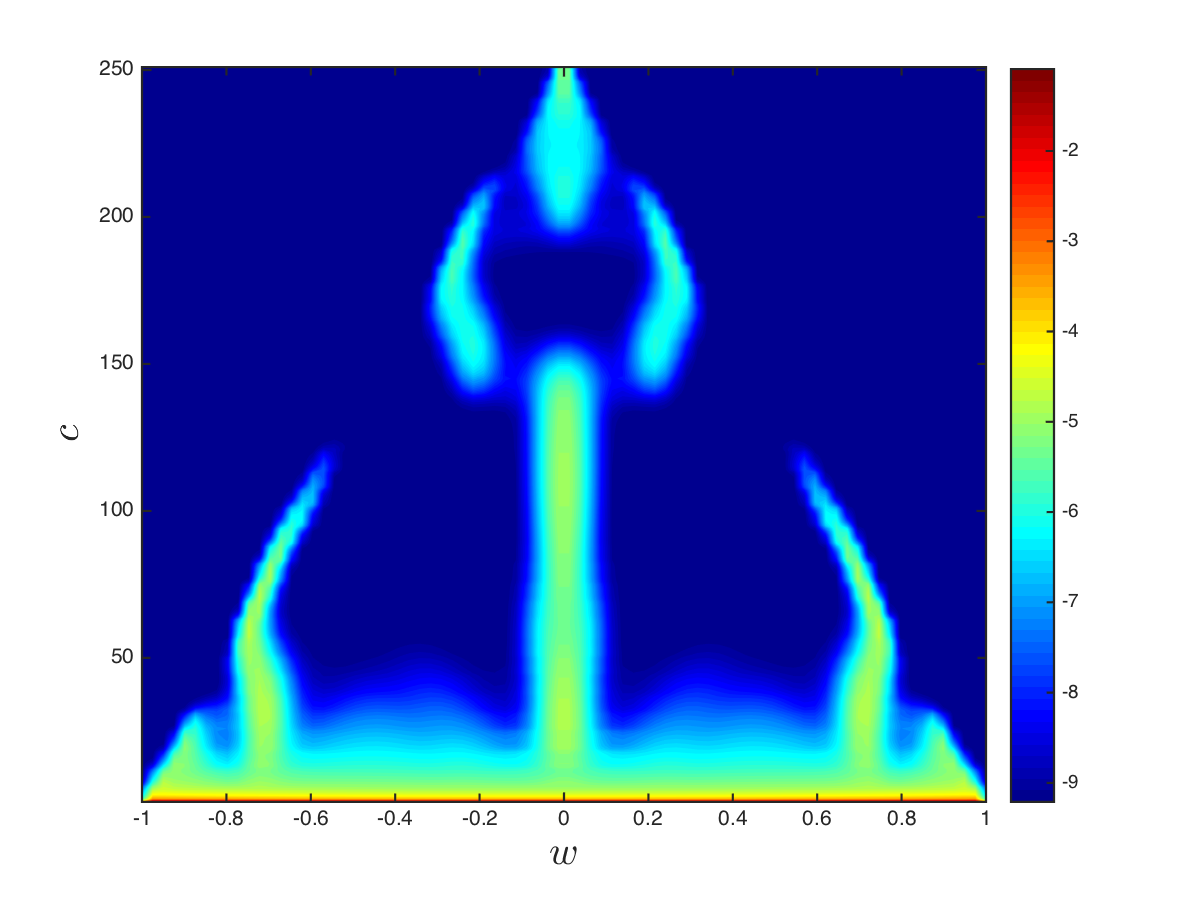}}
\caption{Test 3. Evolution of the solution of the mean--field model \eqref{eq:APZ} with uniform initial opinion and power law type connection distribution. The interaction are described by \eqref{eq:HK} with $d_0=1.01$, in the time interval $[0,100]$. The other parameters are $\sigma^2=10^{-3}$, $c_{\textrm{max}}=250$, $V_r=V_a=1$, $\gamma(0)=30$, $\alpha=10^{-1}$, $\beta=0$. 
}
\end{figure}



\newpage
\begin{thebibliography}{99}

\bibitem{APTZ} 
G.~Albi, L.~Pareschi, G.~Toscani, M.~Zanella. Recent advances in opinion modeling: control and social influence. In \emph{Active Particles Vol.1: Advances in Theory, Models, and Applications}, Birkh{\"a}user--Springer, 2017.

\bibitem{APZ3}
G.~Albi, L.~Pareschi, M.~Zanella. Opinion dynamics over complex networks: kinetic modeling and numerical methods. \emph{Kinetic and Related Models}, 10(1): 1--32, 2017. 

\bibitem{BaDe}
A. B. T. Barbaro, P. Degond. Phase transition and diffusion among socially interacting self-propelled agents. \emph{Discrete and Continuous Dynamical Systems - Series B} 19: 1249--1278, 2014.

\bibitem{BoCar}
F. Bolley, J. A. Carrillo. Stochastic mean--field limit: non--Lipschitz forces and swarming. \emph{Mathematical Models and Methods in Applied Sciences}, 21(11): 2179, 2011. 


\bibitem{BAMT} N.~Bellomo, G.~Ajmone Marsan, A.~Tosin. \emph{Complex Systems and Society. Modeling and Simulation}. SpringerBriefs in Mathematics, Springer, 2013. 

\bibitem{BCS} C.~Buet, S.~Cordier, V.~Dos Santos. A conservative and entropy scheme for a simplified model of granular media. \emph{Transport Theory and Statistical Physics}, 33(2): 125--155, 2004.

\bibitem{BD} C.~Buet, S.~Dellacherie. On the Chang and Cooper numerical scheme applied to a linear Fokker-Planck equation. \emph{Communications in Mathematical Sciences}, 8(4): 1079--1090, 2010.

\bibitem{CCH}
J. A. Carrillo, A. Chertock, Y. Huang. A finite-volume method for nonlinear nonlocal equations with a gradient flow structure. \emph{Communications in Computational Physics} 17: 233--258, 2015.

\bibitem{CFRT}
J. A. Carrillo, M. Fornasier, J. Rosado, G. Toscani. Asymptotic flocking dynamics for the kinetic Cucker--Smale model. \emph{SIAM Journal on Mathematical Analysis}, 42(1): 218--236, 2010.

\bibitem{CFTV}
J. A. Carrillo, M. Fornasier, G. Toscani, F. Vecil. Particle, kinetic and hydrodynamic models of swarming. In \emph{Mathematical modeling of collective behavior in socio-economic and life sciences}, Birkh{\"a}user Boston, pp. 297--336, 2010.



\bibitem{CC}
J. S.  Chang, G. Cooper. A practical difference scheme for Fokker--Planck equations.\emph{Journal of computational physics}, 6(1): 1--16, 1970.

%

\bibitem{CPT} S.~Cordier, L.~Pareschi, G.~Toscani. On a kinetic model for a simple market economy. \emph{Journal of Statistical Physics}, 120(1--2): 253--277, 2005.

\bibitem{CS}
F. Cucker, S. Smale. Emergent behavior in flocks. \emph{IEEE Transactions on automatic control}, 52(5): 852--862, 2007.


\bibitem{DOCBC}
M. R. D'Orsogna, Y. L. Chuang, A. L. Bertozzi, L. Chayes. Self-propelled particles with soft-core interactions: patterns, stability and collapse. \emph{Physical Review Letters} 96, 2006.


\bibitem{FPTT} G. Furioli, A. Pulvirenti, E. Terraneo, G. Toscani. Fokker-Planck equations in the modelling of socio--economic phenomena. \emph{Mathematical Models and Methods for Applied Sciences}, 27(1): 115--158, 2017. 

\bibitem{gosse} {L.~Gosse}, {Computing qualitatively correct approximations of balance laws. Exponential-fit, well-balanced and asymptotic-preserving}.
 SEMA SIMAI Springer Series, Springer--Verlag 2013.

\bibitem{GST}
S. Gottlieb, C. W. Shu, E. Tadmor. Strong stability-preserving high-order time discretization methods. \emph{SIAM Review}, 43(1): 89--112, 2001. 




\bibitem{HNW} {E.~Hairer, S.P.~Norsett, G.~Wanner}, {Solving Ordinary Differential Equation I: Nonstiff Problems}. Springer Series in Comput. Mathematics, Vol. 8, Springer-Verlag 1987, Second revised edition 1993.



\bibitem{LLPS}
E. W. Larsen, C. D. Levermore, G. C. Pomraning, J. G. Sanderson. Discretization methods for one-dimensional Fokker--Planck operators. \emph{Journal of Computational Physics}, 61(3): 359--390, 1985.



\bibitem{PT0}
L. Pareschi, G. Toscani. \emph{Interacting Multiagent Systems: Kinetic Equations and Monte Carlo Methods}, Oxford University Press, 2013.  

\bibitem{PT1} L.~Pareschi, G.~Toscani. Wealth distribution and collective knowledge: a Boltzmann approach. \emph{Philosophical Transactions of the Royal Society of London. Series A. Mathematical, Physical and Engineering Sciences}, 372(2028): 20130396, 2014.

\bibitem{PZ}
L.~Pareschi, M.~Zanella. Structure preserving schemes for nonlinear Fokker-Planck equations and applications. \emph{Preprint}, 2016.
\end{thebibliography}
\end{document}